\documentclass[11pt]{article}

\usepackage{amsmath, amsfonts, amsthm, verbatim}

\newcommand{\R}{\mathbb{R}}
\newcommand{\N}{\mathbb{N}}
\newcommand{\Z}{\mathbb{Z}}
\newcommand{\dd}{\mathrm{d}}
\newcommand{\E}{\mathbf{E}}
\newcommand{\p}{\mathbf{P}}

\theoremstyle{plain}
\newtheorem{lemma}{Lemma}
\newtheorem{theorem}{Theorem}
\newtheorem*{theorem*}{Theorem}
\newtheorem{proposition}{Proposition}
\newtheorem{corollary}{Corollary}

\theoremstyle{remark}
\newtheorem{remark}{Remark}
\newtheorem{example}{Example}

\title{Regularly log-periodic functions and some applications}

\author{P\'eter Kevei \\
MTA-SZTE Analysis and Stochastics Research Group \\
Bolyai Institute, Aradi v\'ertan\'uk tere 1, 6720 Szeged, Hungary \\
\texttt{kevei@math.u-szeged.hu}}

\begin{document}

\maketitle

\begin{abstract}
We prove a Tauberian theorem for the Laplace--Stieltjes transform and 
Karamata-type theorems in the framework of regularly 
log-periodic functions. As an 
application we determine the exact tail behavior of fixed points of certain 
type smoothing transforms.

\end{abstract}

\section{Introduction}

A function $f: [0,\infty) \to [0,\infty)$ is \emph{regularly 
log-periodic}, $f \in \mathcal{RL}$ or $f \in \mathcal{RL}(p,r,\rho)$, if it is 
measurable, there is a slowly varying 
function at infinity $\ell$, real numbers $\rho \in \R$, $r > 1$, and a positive 
logarithmically periodic function $p \in \mathcal{P}_{r}$, such that
\begin{equation} \label{eq:rlp-def}
\lim_{n \to \infty} \frac{f(x r^n)}{(x r^n)^\rho \ell(x r^n)} = p(x), \quad x \in C_p,
\end{equation}
where $C_p$ stands for the set of continuity points of $p$, and for $r > 1$
\[
\begin{split}
\mathcal{P}_{r} = \Big\{  p: (0,\infty) \to (0,\infty)  : & \,
\inf_{x \in [1,r]} p(x) > 0, \ 
p \text{ is bounded, right-continuous, } \\ 
& \text{ and }  p(x r) = p(x), \ 
\forall x >0\Big\}.
\end{split}
\]

This function class is a natural and important extension of regularly varying functions, 
and it appears in different areas of theoretical and applied probability.
This class arises in connection with various random fixed point 
equations, such as the smoothing transformation. Regularly log-periodic functions are the 
basic ingredients in the theory of semistable laws. The tail of the limiting random 
variable of a supercritical Galton--Watson process is also regularly log-periodic. These 
are spelled out in details in Section \ref{sect:appl}. Here we only mention 
some results for the perpetuity equation 
\begin{equation} \label{eq:perp}
X \stackrel{\mathcal{D}}{=} AX + B,
\end{equation}
where $(A,B)$ and $X$ on the right-hand side are independent. Under 
appropriate assumptions, Grincevi\v{c}ius \cite[Theorem 2]{Grinc} showed that the tail of 
the solution of (\ref{eq:perp}) is regularly log-periodic with constant slowly varying 
function. Under similar assumptions the same asymptotic behavior was shown for the 
max-equation
$X \stackrel{\mathcal{D}}{=} \max\{ AX , B\}$, which corresponds to the maximum of 
perturbed random walks; see Iksanov \cite[Theorem 1.3.8]{Iksanov}. More 
generally, this type of tail behavior appears in implicit renewal theory in the 
arithmetic case; see Jelenkovi\'c and Olvera-Cravioto \cite[Theorem 3.7]{JOC3}, 
and Kevei \cite{Kevei2}.
In general, functions of the form $p(x) e^{\lambda x}$, $\lambda \in 
\R$, where $p$ is a periodic function, are solutions of certain integrated Cauchy 
functional equations, see Lau and Rao \cite{LauRao}.

The name `regularly log-periodic' comes from Buldygin and Pavlenkov \cite{BP, BP2}, where 
a function $f$ is called regularly log-periodic, if
\begin{equation} \label{eq:f-def}
f(x) = x^\rho \ell(x) p(x), \quad x > 0,
\end{equation}
where $\ell, \rho$ and $r$ are the same as above, and $p \in \mathcal{P}_r$ is 
\emph{continuous}. This condition is clearly much stronger than (\ref{eq:rlp-def}) even 
without the continuity of $p$. In the examples 
given above, the continuity assumption does not necessarily hold, and this is the reason 
for the extension of the definition. Moreover, our main motivation originates in the 
studies of the St.~Petersburg distribution, where the corresponding $p$ 
function is not continuous; see Example \ref{ex:StP} at the end of Subsection 
\ref{subsect:tails}.

\smallskip

In what follows, we assume that $U:[0,\infty) \to [0,\infty)$ is a nondecreasing 
function, and
\[
\widehat U(s) = \int_0^\infty e^{-sx} \dd U(x)
\]
denotes its Laplace--Stieltjes transform.
Since we need monotonicity, for $r > 1$ we further introduce the sets of functions 
\begin{equation} \label{eq:def-P}
\begin{split}
&\mathcal{P}_{r,\rho} = \Big\{  p: (0,\infty) \to (0,\infty) \, : \,
p \in \mathcal{P}_{r}, \text{ and } x^{\rho} p(x) \text{ is nondecreasing} \Big\}, 
\ \rho \geq 0, \\
&\mathcal{P}_{r,\rho} = \Big\{  p: (0,\infty) \to (0,\infty) \, : \,
p \in \mathcal{P}_{r}, \text{ and } x^{\rho} p(x) \text{ is nonincreasing} \Big\}, 
\ \rho < 0.
\end{split}
\end{equation}
In order to characterize the Laplace--Stieltjes transform of regularly log-periodic 
functions, for $r > 1$, $\rho \geq 0$, put
\begin{equation} \label{eq:def-Q}
\mathcal{Q}_{r,\rho} = \Big\{  q: (0, \infty) \to (0,\infty) \, :
\, s^{-\rho} q(s) \text{ is completely monotone, and } q(s r) = q(s), \ \forall s>0  
\Big\}.
\end{equation}
For $\rho = 0$ the sets 
$\mathcal{P}_{r, 0}, \mathcal{Q}_{r, 0}$ are just the set of constant functions.

\smallskip

The aim of the present paper is to prove Tauberian theorem for the Laplace--Stieltjes 
transform, and Karamata-type theorems in the framework of regularly log-periodic 
functions. The ratio Tauberian theorem \cite[Theorem 2.10.1]{BGT}, a general version of 
the Tauberian theorem for Laplace-Stieltjes transforms, holds for O-regular varying 
functions. The equivalence of the behavior of $U$ at infinity and $\widehat U$ at zero 
holds, if and only if $U^*(\lambda) = \limsup_{x \to \infty} U(\lambda x)/ U(x)$ is 
continuous at 1. The latter condition for functions defined in (\ref{eq:f-def}) is 
equivalent to the continuity of $p$; see Proposition \ref{prop:rlp}. 
In particular, the discontinuity of $p$ is the reason that the ratio Tauberian theorem 
\cite[Theorem 2.10.1]{BGT} does not hold in this setup. However, in Theorem \ref{thm:taub} 
below we do 
provide an equivalence between the tail behavior of the function, and the behavior of 
its Laplace--Stieltjes transform at zero. 
In  \cite{BP, BP2}, Buldygin and Pavlenkov proved 
Karamata theorems in the sense of Theorems 1.5.11 
(direct half) and 1.6.1 (converse half) of Bingham, Goldie and Teugels \cite{BGT}, for 
functions satisfying (\ref{eq:f-def}) with continuous $p$.
Here we extend these results.

\smallskip

Section \ref{sect:results} contains the main results of 
the paper. After some preliminaries, first we deal with a Tauberian theorem for the 
Laplace--Stieltjes transform, 
then we prove the direct half of the Karamata theorem, and a monotone density theorem. 
In Section \ref{sect:appl} we give some applications. We prove that the tail of 
a 
nonnegative random variable is regularly log-periodic, if and only if the same is true 
for its Laplace transform at 0. Using this result we determine the tail behavior of 
fixed points of certain smoothing transforms. We reprove, in a special case, a 
result by Watanabe and Yamamuro \cite{WY3} for tails of semistable random variables. 
Finally, we spell out some related results on the limit of supercritical branching 
processes.

\section{Results} \label{sect:results}

\subsection{Preliminaries}

First we discuss the place of the regularly log-periodic functions among well-known 
function classes, such as regularly varying functions, extended and O-regularly varying 
functions.

In the following we always assume that $f: [0,\infty) \to [0, \infty)$ is nonnegative and 
measurable. For $\lambda > 0$ let
\[
f^*(\lambda) = \limsup_{x \to \infty} \frac{f(\lambda x)}{f(x)}, \quad
f_*(\lambda) = \liminf_{x \to \infty} \frac{f(\lambda x)}{f(x)}.
\]
A function $f$ is \emph{extended regularly varying} if for some constants $c, d$
\begin{equation} \label{eq:ER-def}
\lambda^d \leq f_*(\lambda) \leq f^*(\lambda) \leq \lambda^c, \ \lambda > 1,
\end{equation}
and it is \emph{O-regularly varying} if
\[
 0 < f_*(\lambda) \leq f^*(\lambda) < \infty.
\]

First we note that general regularly log-periodic functions can be quite irregular.

\begin{example}
Consider the function
\begin{equation} \label{eq:counter}
f(x) =
\begin{cases}
n, & \text{if } x \in [ (1 + n^{-1}) 2^n,  (1 + 2n^{-1}) 2^n], \\
1, & \text{otherwise.}
\end{cases}
\end{equation}
Then (\ref{eq:rlp-def}) holds with $\ell(x) \equiv 1$, $\rho=0$, $r=2$, and $p(x) \equiv 
1$. Indeed, $\lim_{n \to \infty} f(2^n x) = 1$ for every $x > 0$, but $f$ is not even 
bounded, and the exceptional intervals are large.
\end{example}

For monotone log-periodic functions the situation is not so bad. A function $f:[0,\infty) 
\to [0, \infty)$ is ultimately monotone if it is monotone (increasing or decreasing) for 
large enough $x$.

\begin{proposition} \label{prop:rlp0}
Let $f \in \mathcal{RL}(p, r, \rho)$ be an ultimately monotone regularly log-periodic 
function. Then
\[
\limsup_{x \to \infty} \frac{f(x)}{x^\rho \ell(x)} < \infty,
\]
and $f$ is O-regularly varying.
\end{proposition}

\begin{proof}
Assume that $f$ is ultimately monotone increasing. The decreasing case follows the same 
way. Indirectly assume that $f(x_n)/(x_n^{\rho } \ell(x_n)) \to \infty$ for some 
$x_n \uparrow \infty$. Write $x_n = r^{k_n} z_n$, where $z_n \in [1,r)$. Using the 
Bolzano--Weierstrass theorem, we may assume that $z_n \to \lambda \in [1,r]$. With some 
$\lambda < \eta \in C_p$, for large enough $n$
\[
\frac{f( r^{k_n} z_n )}{(r^{k_n} z_n)^{\rho } \ell(r^{k_n} z_n)} \leq 
\frac{f( r^{k_n} \eta )}{(r^{k_n})^{\rho } \ell(r^{k_n} z_n) } \to \eta^{\rho} 
p(\eta), 
\]
which is a contradiction. The O-regular variation follows from the boundedness 
and strict positivity of $p$.
\end{proof}

For the extended regular variation, and for the continuity of $f^*$ stronger conditions 
are needed.

\begin{proposition} \label{prop:rlp}
Assume that for a slowly varying function $\ell$, for $\rho \in \R$, $r > 1$, and $p \in 
\mathcal{P}_r$ 
\[
f(x) = x^\rho \ell(x) p(x).
\]
Then $f$ is
\begin{itemize}
\item[(i)] extended regularly varying if and only if $p$ is Lipschitz on 
$[1,r]$;
\item[(ii)] regularly varying if and only if $p$ is constant.
\end{itemize}
Moreover, $f^*$ is continuous at 1, if and only if $p$ is continuous. 
\end{proposition}

Note that a logarithmically periodic function is globally Lipschitz if and 
only if it is constant.

\begin{proof}[Proof of Proposition \ref{prop:rlp}]
The logarithmic periodicity of $p$ implies
\[
f^*(\lambda) = \lambda^\rho \sup_{x \in [1,r]} \frac{p(\lambda x)}{p(x)}, 
\]
from which we see that $f^*$ is continuous at 1 if and only if $p$ is continuous.

We turn to (i). Let $\lambda > 1$. If $p$ is Lipschitz with Lipschitz constant $L$, 
then for $x \in [1,r]$ we have
$p(\lambda x) \leq p(x) + L x (\lambda -1)$, thus
\[
\sup_{x \in [1,r]} \frac{p(\lambda x)}{p(x)} \leq 
1 + L (\lambda -1) \sup_{x \in [1,r]} \frac{x}{p(x)} \leq \lambda^{c - \rho}
\]
for some $c > 0$. The proof of the lower bound is similar. For the converse, assume 
indirectly that $p$ is not Lipschitz. Then there are two sequences $\lambda_n \downarrow 
1$, and $x_n \to x \in [1,r]$ such that
\[
|p(\lambda_n x_n ) - p(x_n) | \geq n x_n (\lambda_n -1),
\]
consequently (\ref{eq:ER-def}) cannot hold. Finally, (ii) is obvious.
\end{proof}

\subsection{Tauberian theorem for the Laplace transform}

Recall (\ref{eq:def-P}) and (\ref{eq:def-Q}).
There is a natural correspondence between $\mathcal{P}_{r,\rho}$ and 
$\mathcal{Q}_{r,\rho}$.

\begin{lemma} \label{lemma:Laplace}
For $p \in \mathcal{P}_{r,\rho}$, $\rho > 0$, define the operator $\mathrm{A}_{r, \rho} = 
\mathrm{A}_\rho$ as
\begin{equation} \label{eq:def-tildeq}
\mathrm{A}_\rho p (s) = s^{\rho} \int_0^\infty e^{-s x} \dd ( p(x) x^\rho) .
\end{equation}
Then $\mathrm{A}_\rho : \mathcal{P}_{r,\rho} \to \mathcal{Q}_{r,\rho}$ is one-to-one.
\end{lemma}

\begin{proof}[Proof of Lemma \ref{lemma:Laplace}]
It is clear from the definition that $\mathrm{A}_\rho p \in \mathcal{Q}_{r,\rho}$.

Conversely, let $q \in \mathcal{Q}_{r,\rho}$ be given. 
Since $s^{-\rho} q(s)$ is completely monotone, there is a nondecreasing 
right-continuous function $g:[0,\infty) \to [0, \infty)$, $g(0) = 0$ 
such that 
\begin{equation} \label{eq:def-g}
s^{-\rho} q(s) = \int_0^\infty e^{-sx} \dd g(x) . 
\end{equation}
To prove that $p(x) : =x^{-\rho} g(x) \in \mathcal{P}_{r,\rho}$ we only have to show 
the logarithmic periodicity of $p$. Substituting $s \to rs$ in (\ref{eq:def-g}) and 
using that 
$q(rs) = q(s)$ we obtain that
\[
\int_0^\infty e^{-sx} \dd g(x) = \int_0^\infty e^{-sx} \dd [ r^\rho g(x/r)].  
\]
Uniqueness of the Laplace--Stieltjes transform implies
\[
 g(x) = r^{\rho} g(x/r), \quad x \in C_g,
\]
from which
\[
 p(x) = p(x/r), \quad x \in C_p.
\]
If two right-continuous functions agree in all but countable many points, then they agree 
everywhere.
\end{proof}

For a real function $f$ the set of its continuity points is denoted by $C_f$. 
In the following, $\ell$ stands for a slowly varying function either at infinity, or at 
zero. The set of slowly varying functions at infinity (zero) is denoted by 
$\mathcal{SV}_\infty$ ($\mathcal{SV}_0$).

\begin{theorem} \label{thm:taub}
Let $U : [0, \infty) \to [0, \infty)$ be an increasing function, $\rho \geq 0$, $r > 1$, 
and $\ell \in \mathcal{SV}_\infty$ be a slowly varying function. Then
\begin{equation} \label{eq:U-asy}
\lim_{n \to \infty} \frac{U(r^n z) }{(r^n z)^\rho \ell(r^n z)} = p(z)
\quad \text{for each } z \in C_p, \text{ for some } p \in \mathcal{P}_{r},
\end{equation}
and
\begin{equation} \label{eq:hatU-asy}
\widehat U(s) \sim s^{-\rho} \ell(1/s) q(s)
\quad \text{as } s \downarrow 0, \text{ for some } q \in \mathcal{P}_{r},
\end{equation}
are equivalent. In each case, necessarily $p \in \mathcal{P}_{r, \rho}$, $q \in 
\mathcal{Q}_{r, \rho}$, and $\mathrm{A}_\rho p = q$ for $\rho > 0$, and $p = q$ for $\rho 
= 0$.

Moreover, if $p$ is continuous, then (\ref{eq:U-asy}) implies
\begin{equation} \label{eq:U-asy2}
U(x) \sim x^\rho \ell(x) p(x) \quad \text{as } x \to \infty.
\end{equation}
\end{theorem}

\begin{remark}
\begin{itemize}
\item[(i)] For $\rho = 0$ the result follows from \cite[Theorem 1.7.1]{BGT}.

\item[(ii)] The equivalence of $U(r^n z) = o(r^n \ell(r^n))$ and $\widehat U(s) = 
o(s^{-\rho} \ell(1/s))$ also follows from \cite[Theorem 1.7.1]{BGT}.

\item[(iii)] For continuous $p$ the ratio Tauberian theorem \cite[Theorem 2.10.1]{BGT}, 
(Korenblyum \cite{Koren}, Feller \cite{Fe}, Stadtm\"uller and Trautner \cite{StT}) states 
that (\ref{eq:hatU-asy}) and (\ref{eq:U-asy2}) are equivalent.
Indeed, by Propositions \ref{prop:rlp0} and \ref{prop:rlp} $U$ is always 
O-regularly varying and $p$ is 
continuous if and only if 
$U^*(\lambda)$ is continuous at 1. Moreover, the Laplace--Stieltjes transform of 
$x^\rho p(x)$ is $s^{-\rho} q(s)$. Theorem 2.10.1 (iii) \cite{BGT} states that the 
continuity of $U^*$ at 1, is also necessary in general for the equivalence of 
(\ref{eq:hatU-asy}) and (\ref{eq:U-asy2}).
\end{itemize}
\end{remark}

\begin{proof}[Proof of Theorem \ref{thm:taub}]
Concerning the first remark above, we may assume that $\rho > 0$.
The proof follows the standard idea of Tauberian theorems (see Theorem 1.7.1 
\cite{BGT}) combined with the following lemma from \cite{Kevei2}.

\begin{lemma} \label{lemma:cont-q}
Assume that $p \in \mathcal{P}_{r}$ is continuous, $\ell \in \mathcal{SV}_\infty$,
$\alpha \in \R$, $U$ is monotone, and
for any $z \in [1,r)$
\[
\lim_{n \to \infty} \frac{U(z r^n)}{(z r^n)^\alpha \ell(r^n)} = p(z).
\]
Then
\[
U(x) \sim x^\alpha \ell(x) p(x). 
\]
\end{lemma}

The monotonicity of $U$ and (\ref{eq:U-asy}) readily imply that $p \in 
\mathcal{P}_{r,\rho}$. From Proposition \ref{prop:rlp0}
\begin{equation} \label{eq:U-bound}
\limsup_{x \to \infty} \frac{U(x)}{x^\rho \ell(x) } = K < \infty 
\end{equation}
follows. Using Potter's bounds we obtain
\[
\begin{split}
\widehat U(x^{-1}) & = \int_0^\infty e^{-y/x} \dd U (y) \\
& \leq U(x) + \sum_{n=1}^{\infty} e^{-2^{n-1}} U(2^n x) \\
& \leq 2 K x^{\rho} \ell(x) \left[ 1 + \sum_{n=1}^\infty 
e^{-2^{n-1}} 2^{n(\rho +1)} \right]. 
\end{split}
\]
Therefore $\widehat U(x^{-1}) / (x^\rho \ell(x))$ is bounded. Introduce the notation
\[
U_x (y) = \frac{U(xy)}{x^\rho \ell(x)}.
\]
Using the logarithmic periodicity, for any $z > 0$ we have
\[
\lim_{n \to \infty} U_{r^n z}(y) = y^\rho p(zy) =: V_z(y)
\quad \text{for all $y$ such that } zy \in C_p.
\]
Simply
\[
\widehat U_x(s) = \frac{\widehat U(s/x)}{x^{\rho} \ell(x) }.
\]
Since 
$U_{r^n z}(y)$ converges, using the continuity and uniqueness theorem for 
Laplace--Stieltjes transforms, we obtain that
\[
\lim_{n \to \infty} \frac{\widehat U(s/(r^n z))}{(r^n z)^\rho \ell (r^n z) }
= \widehat V_z(s)
\]
for all $s >0$, since $\widehat V_z$, being a Laplace--Stieltjes transform, is continuous.
Choosing $s=1$, after short calculation we have
\[
\lim_{n \to \infty} 
\frac{\widehat U(1/(r^n z))}{(r^n z)^\rho \ell (r^n z)} =  q(1/z),
\]
with $q = \mathrm{A}_\rho p$.
Since $q$ is continuous, Lemma \ref{lemma:cont-q} implies
\[
\widehat U(s) \sim  s^{-\rho} \ell(1/s) q(s) \quad \text{as } s \downarrow 
0,
\]
as stated.

For the converse, note that (\ref{eq:hatU-asy}) implies
\[
\widehat U_x (s) = \frac{\widehat U(s/x)}{x^\rho \ell(x)}
\sim  s^{-\rho} q(s/x)
\quad \text{as } x \to \infty.
\]
Since $q \in \mathcal{P}_{r}$ we have for any $z >0$
\begin{equation} \label{eq:hatUx}
\lim_{n \to \infty} \widehat U_{r^n z}(s) =  s^{-\rho} q(s/z).
\end{equation}
Therefore, the continuity theorem gives
\[
\lim_{n \to \infty} U_{r^n z}(y) = u_z(y), \quad y \in C_{u_z}
\]
for some nondecreasing function $u_z$. Thus 
$\widehat u_z(s) =  s^{-\rho} q(s/z)$, which implies $q \in \mathcal{Q}_{r,\rho}$.
Short calculation shows that the right-hand side of (\ref{eq:hatUx}) is the 
Laplace--Stieltjes transform of $u_z(y) :=  y^\rho p(zy)$.
Note that $1 \in C_{u_z}$ whenever $z \in C_p$, thus
(\ref{eq:U-asy}) holds for all $z \in C_p$. The second statement follows from Lemma 
\ref{lemma:cont-q}.
\end{proof}

The same proof gives analogous result in the case $x \downarrow 0$, $s \to \infty$; see
\cite[Theorem 1.7.1']{BGT}.

\subsection{Karamata and monotone density theorems}

Let $\mathcal{P}_{r,\rho}^{m}$ denote the set of functions in $\mathcal{P}_{r,\rho}$, 
which are $m$-times differentiable on $(0,\infty)$ (we do not assume continuity of the 
$m$th derivative).
For $r > 1$ and $\rho > 0$ introduce the operator 
$\mathrm{B}_{r, \rho} = \mathrm{B}_{\rho}: \mathcal{P}_{r} \to 
\mathcal{P}_{r,\rho}^{1}$
\begin{equation} \label{eq:defB}
\mathrm{B}_\rho p (x) = x^{-\rho} \int_0^x y^{\rho -1 } p(y) \dd y. 
\end{equation}
Using the logarithmic periodicity, short calculation shows that
\[
\int_0^{r^m} s^{\rho -1} p(s) \dd s = \frac{r^{m\rho}}{r^\rho -1} \int_1^r s^{\rho-1} 
p(s) 
\dd s, 
\]
and thus
\begin{equation} \label{eq:Bp2form}
\mathrm{B}_\rho p(x) =  r^{-\rho \{ \log_r x \}} \left[ 
\frac{1}{r^\rho -1} \int_1^r s^{\rho-1} p(s) \dd s +
\int_1^{r^{\{ \log_r x \}}} s^{\rho-1} p(s) \dd s \right],
\end{equation}
where $\{ x \} = x - \lfloor x \rfloor$ stands for the fractional part of $x$.
It is easy to see that $\mathrm{B}_\rho p \in \mathcal{P}_{r,\rho}^{1}$. Moreover, it is 
one-to-one 
with inverse
\begin{equation} \label{eq:Binv}
{\mathrm{B}}_\rho^{-1} q (x) =  x^{1-\rho} \frac{\dd}{\dd x} [x^\rho q(x)], \quad q \in 
\mathcal{P}_{r,\rho}^{1}.
\end{equation}

The following statement is a Karamata type theorem for regularly log-periodic functions; 
see \cite[Theorem 1.5.11]{BGT}.

\begin{theorem} \label{thm:Karamata}
Assume that for some $\rho > 0$,
\begin{equation} \label{eq:u-asy}
\lim_{n \to \infty} \frac{u(r^n z)}{(r^n z)^{\rho -1} \ell(r^n z)} = p_0(z) 
\quad \text{for each } z \in C_{p_0},
\text{ for some } p_0 \in \mathcal{P}_{r},
\end{equation}
and 
\begin{equation} \label{eq:limsup-u}
\limsup_{x \to \infty} \frac{u(x)}{x^{\rho -1} \ell(x)} < \infty. 
\end{equation}
Then
\begin{equation} \label{eq:utoU}
U(x) = \int_0^x u(y) \dd y \sim x^\rho \ell(x) p(x) \quad \text{as } x \to \infty, 
\end{equation}
where $p= \mathrm{B}_\rho p_0$.
\end{theorem}

\begin{remark} \label{rem:Karamata}
\begin{itemize}
\item[(i)] For continuous $p_0$ condition
\[
u(x) \sim x^{\rho-1} \ell(x) p_0(x) \quad \text{as } x \to \infty,
\]
implies (\ref{eq:utoU}); see Lemma 3 by Buldygin and Pavlenkov 
\cite{BP2}. (Compare with formula (\ref{eq:Bp2form}). Note that 
our $\rho$ and their $\rho$ are different.) 

\item[(ii)] It is again straightforward to extend this result to the case when the limit 
in (\ref{eq:u-asy}) is zero.
\end{itemize}
 
\end{remark}

\begin{proof}[Proof of Theorem \ref{thm:Karamata}]
From (\ref{eq:limsup-u}) we readily obtain as in \cite[Proposition 1.5.8]{BGT} that
\begin{equation} \label{eq:limsup-U}
\limsup_{x \to \infty} \frac{U(x)}{x^\rho \ell(x)} < \infty. 
\end{equation}
Short calculation gives for any $0 < \varepsilon < 1$
\[
\frac{U( r^n z) - U(r^n z \varepsilon) }{(r^n z)^\rho \ell(r^n z) } = 
\int_{\varepsilon}^1 \frac{u(r^n z t)}{(r^n z t )^{\rho -1} \ell(r^n z t)}
t^{\rho -1} \frac{\ell(r^n z t)}{\ell(r^n z)} \dd t.
\]
Whenever $t z \in C_p$ the integrand converges to $p_0(tz) t^{\rho -1}$. Since the set of 
discontinuity points of a right-continuous function is at most countable, and integrable 
majorant exists by (\ref{eq:limsup-u}) we see
\[
\lim_{n \to \infty} \frac{U( r^n z) - U(r^n z \varepsilon) }{(r^n z)^\rho \ell(r^n z) }
= \int_\varepsilon^1 t^{\rho -1} p_0(tz) \, \dd t.
\]
Finally, (\ref{eq:limsup-U}) implies
\[
\limsup_{\varepsilon \downarrow 0} \limsup_{n \to \infty}  
\frac{ U(r^n z \varepsilon) }{(r^n z)^\rho \ell(r^n z) } = 0.
\]
Combining the latter two limit relations we obtain
\begin{equation} \label{eq:U-limits}
\lim_{n \to \infty} \frac{U( r^n z) }{(r^n z)^\rho \ell(r^n z) }
= \int_0^1 t^{\rho -1} p_0(tz) \dd t = z^{-\rho} \int_0^z s^{\rho -1} p_0(s) \dd s =
\mathrm{B}_\rho p_0(z).
\end{equation}
Since $\mathrm{B}_\rho p_0$ is continuous, the statement follows from Lemma 
\ref{lemma:cont-q}.
\end{proof}

The statement remains true for $\rho = 0$ in the following version.

\begin{lemma} \label{lemma:Karamatar=0}
Assume that for some $p_0 \in \mathcal{P}_{r}$
\begin{equation} \label{eq:u-asyr=0}
\lim_{n \to \infty} \frac{r^n z \, u(r^n z)}{\ell(r^n z)} = p_0(z) 
\quad \text{for each } z \in C_{p_0},
\end{equation}
and
\begin{equation} \label{eq:liminfsup-u}
0 < \liminf_{x \to \infty} \frac{x u(x)}{\ell(x)}  \leq
\limsup_{x \to \infty} \frac{x u(x)}{\ell(x)} < \infty. 
\end{equation}
Then $U(x) = \int_0^x u(y) \dd y$ is slowly varying, and 
$\lim_{x \to \infty} U(x) / \ell(x) = \infty$. 
\end{lemma}

\begin{remark}
As for Theorem \ref{thm:Karamata}, condition (\ref{eq:liminfsup-u}) is 
not very restrictive, and necessary in general.
\end{remark}

\begin{proof}
The proof is almost identical to the proof of \cite[Proposition 1.5.9a]{BGT}.

Put
\[
\liminf_{x \to \infty} \frac{x u(x) }{\ell(x)} =: k > 0.
\]
Then
\[
\liminf_{x \to \infty} \frac{U(x)}{\ell(x)} \geq  \frac{k}{2} \liminf_{x \to \infty} 
\frac{1}{\ell(x)} \int_{\varepsilon x}^x \frac{\ell(y)}{y} \dd y = \frac{k}{2} \log 
\varepsilon^{-1}.
\]
As $\varepsilon \downarrow 0$ we get $\lim_{x \to \infty} U(x) / \ell(x) = \infty$. Put 
$\varepsilon(x) = x u(x) / U(x)$. We showed that $\lim_{x \to \infty} \varepsilon(x) = 
0$. Noticing
\[
\frac{\dd}{\dd x} \log U(x) = \frac{U'(x)}{U(x)} = \frac{\varepsilon(x)}{x},
\]
the representation theorem of slowly varying functions (\cite[Theorem 1.3.1]{BGT}) 
finishes the proof.
\end{proof}

The converse part of Theorem \ref{thm:Karamata} is the corresponding monotone density 
theorem. 

\begin{theorem} \label{thm:mondens}
Assume that $U(x) = \int_0^x u(y) \dd y$, $u$ is ultimately monotone, and 
(\ref{eq:U-asy}) holds with $\rho \geq 0$. If $\rho > 0$, then $p = \mathrm{B}_\rho p_0$ 
for some $p_0 \in \mathcal{P}_{r}$. For $\rho =0$ let $p_0(x) \equiv 0$.  In both 
cases
\[
\lim_{n \to \infty} \frac{u(r^n z)}{(r^n z)^{\rho -1} \ell(r^n z)} =
p_0(z) \quad \text{for each } z \in C_{p_0}.
\]
Moreover, if $p_0$ is continuous, then
\[
u(x) \sim x^{\rho-1} \ell(x) p_0(x) \quad \text{as } x \to \infty. 
\]
\end{theorem}

\begin{remark}
\begin{itemize}
\item[(i)] We see from the statement that if (\ref{eq:U-asy}) holds, and $U$ has an 
ultimately monotone density, then necessarily  $p$ in (\ref{eq:U-asy}) is differentiable. 

\item[(ii)] Note that for $\rho = 0$ the statement follows from the `usual' monotone 
density theorem  \cite[Theorem 1.7.2]{BGT}, since $p \in \mathcal{P}_{r}$ is necessarily 
constant. Theorem 1.7.2 \cite{BGT} also implies that the result remains true when 
the limit $p$ in (\ref{eq:U-asy}) is zero.
\end{itemize}
\end{remark}

\begin{proof}[Proof of Theorem \ref{thm:mondens}]
By (\ref{eq:U-asy})
\[
\frac{U(bx) - U(ax)}{x^\rho \ell(x)} = \int_a^b \frac{u(sx)}{x^{\rho-1} \ell(x)}  \dd s
\]
is bounded as $x \to \infty$. Since $u$ is ultimately monotone, this readily 
implies that the integrand is bounded too as $x \to \infty$, which allows us to use 
Helly's selection theorem. Fix $z > 0$, and consider the sequence $r^n z$. By the 
selection theorem, there is a subsequence 
$n_k$ and a monotone limit function $v_z$ such that
\begin{equation} \label{eq:uconv-ss}
\lim_{k \to \infty} \frac{u(r^{n_k} z s)}{(r^{n_k} z)^{\rho-1} \ell(r^{n_k} z)} = v_z(s) 
\quad \text{for each } s \in C_{v_z}.
\end{equation}
On the other hand, $U(xy)/(x^\rho \ell(x))$ converges on the sequence $r^n z$, thus 
for the limit function $v_z$
\begin{equation} \label{eq:v-zdet}
\int_a^b v_z(s) \dd s =   b^\rho p(b z) - a^\rho p(a z)
\end{equation}
for $0 < a < b < \infty$ such that $az, bz \in C_p$. This clearly determines the limit 
function in its continuity points, and so
the convergence in (\ref{eq:uconv-ss}) holds along the whole sequence $n$. The latter 
implies that $v_z(rs) = r^{\rho -1} v_z(s)$. From (\ref{eq:v-zdet}) we have that $p \in 
\mathcal{P}_{r,\rho}^{1}$. Let $p_0 = \mathrm{B}_\rho^{-1} p$. By (\ref{eq:Binv})
\[
v_z(s) =  \frac{\dd}{\dd s} (s^\rho p(s z)) =  s^{\rho-1} p_0(s z).
\]
If $z \in C_{p_0}$, then $s =1$ is a continuity point of $v_z$ in (\ref{eq:uconv-ss}), 
and the first statement follows. The second follows from Lemma \ref{lemma:cont-q}.
\end{proof}

The following statements are versions of the previous results, which we need later. Since 
the proofs are the same, we omit them.

First we deal with the case when $\rho < 0$. Similarly as before let 
$\mathcal{P}_{r,\rho}^{1}$ denote the set of functions in $\mathcal{P}_{r,\rho}$, which 
are
differentiable on $(0,\infty)$. For $r > 1$ and $\rho < 0$ introduce the operator 
$\mathrm{B}_{r,\rho} = \mathrm{B}_\rho: \mathcal{P}_{r} \to 
\mathcal{P}_{r,\rho}^{1}$
\begin{equation} \label{eq:defB2}
\mathrm{B}_\rho p (x) = x^{-\rho} \int_x^\infty y^{\rho -1 } p(y) \dd y. 
\end{equation}
As before $\mathrm{B}_\rho p \in \mathcal{P}_{r,\rho}^{1}$, and it is one-to-one 
with inverse
\begin{equation} \label{eq:Binv2}
{\mathrm{B}}^{-1}_\rho q (x) =  - x^{1-\rho} \frac{\dd}{\dd x} [x^\rho q(x)], \quad q \in 
\mathcal{P}_{r,\rho}^{1}.
\end{equation}

\begin{proposition} \label{prop:rho<0}
Let $U(x) = \int_x^\infty u(y) \dd y$, where $u$ is ultimately monotone, $r > 1, \rho < 
0$.
Then
\[
\lim_{n \to \infty} \frac{u(r^n z)}{(r^n z)^{\rho -1} \ell(r^{n} z)} =
p_0(z) \quad \text{for each } z \in C_{p_0},
\text{ for some } p_0 \in \mathcal{P}_{r},
\]
if and only if
\[
\lim_{n \to \infty} \frac{U(r^n z)}{(r^n z)^{\rho} \ell(r^n z)} =
p(z) \quad \text{for each } z \in C_{p},
\text{ for some } p \in \mathcal{P}_{r}.
\]
Moreover, $p = \mathrm{B}_\rho p_0$, in particular $p \in \mathcal{P}_{r,\rho}$ is 
continuous, thus
\[
U(x)  \sim x^{\rho} \ell(x) p(x) \quad \text{as } x \to \infty.
\]
For $\rho = 0$ assume further that $\int_0^\infty u(y) \dd y < \infty$. Then $U \in 
\mathcal{SV}_\infty$, and $\lim_{x \to \infty} U(x) / \ell(x) = \infty$.
\end{proposition}

For continuous $p$ see \cite[Lemma 3]{BP2}.

At 0 the corresponding result is the following.

\begin{proposition} \label{prop:Karamata0}
Let $U(x) = \int_0^x u(y) \dd y$, where $u$ is ultimately monotone, $r > 1, \rho > 0$, 
and 
$\ell \in \mathcal{SV}_0$. Then
\[
\lim_{n \to \infty} \frac{u(r^{-n} z)}{(r^{-n} z)^{\rho -1} \ell(r^{-n} z)} =
p_0(z) \quad \text{for each } z \in C_{p_0},
\text{ for some } p_0 \in \mathcal{P}_{r},
\]
if and only if
\[
\lim_{n \to \infty} \frac{U(r^{-n} z)}{(r^{-n} z)^{\rho} \ell(r^{-n} z)} =
p(z) \quad \text{for each } z \in C_{p},
\text{ for some } p \in \mathcal{P}_{r}.
\]
Moreover, $p = \mathrm{B}_\rho p_0$, in particular $p$ is continuous, thus
\[
U(x)  \sim x^{\rho} \ell(x) p(x) \quad \text{as } x \downarrow 0.
\]
\end{proposition}

\section{Applications} \label{sect:appl}

\subsection{Tails of nonnegative random variables} \label{subsect:tails}

In this subsection we prove the log-periodic analogue of Theorem A by Bingham 
and Doney \cite{BD}  (Theorem 8.1.8 in \cite{BGT}).

Let $X$ be a nonnegative random variable with distribution function $F$. If 
$\E X^m < \infty$, then its Laplace transform 
\begin{equation} \label{eq:hatF}
\widehat F(s) = \int_0^\infty e^{-sx} \dd F(x) 
\end{equation}
can be written as
\[
\widehat F(s) = \sum_{k=0}^m \mu_k \frac{(-s)^k}{k!} + o(s^m) \quad \text{as } s 
\downarrow 0,
\]
where $\mu_k = \E X^k$. Define for $m \geq 0$
\begin{equation} \label{eq:def-fg}
\begin{split}
f_m(s) & =  (-1)^{m+1} \left[ \widehat F(s) - \sum_{k=0}^m \mu_k \frac{(-s)^k}{k!} 
\right], \\
g_m(s) & = \frac{\dd^m}{\dd s^m} f_m(s) = \mu_m - (-1)^{m} \widehat F^{(m)}(s).
\end{split}
\end{equation}

\begin{theorem} \label{thm:L-tail}
Let $\ell \in \mathcal{SV}_\infty$, $m \in \{0,1,\ldots\}$, $\alpha = m + \beta$, 
$\beta \in [0,1]$, $\tilde q_m, q_m, p \in \mathcal{P}_{r}$.
The following are equivalent:
\begin{align} 
& f_m(s) \sim s^\alpha \ell(1/s) \tilde q_m(s) \text{ as } s \downarrow 0; 
\label{eq:Ltail-1}\\
& g_m(s) \sim s^\beta \ell(1/s) q_m(s) \text{ as } s \downarrow 0; 
\label{eq:Ltail-2}\\
& 
\begin{cases}
\lim_{n \to \infty} \ell(r^n)^{-1} \int_{r^n z}^\infty y^m \dd F(y) 
= p(z) \text{ for each } z \in C_p, & \beta = 0, \\
\lim_{n \to \infty} \frac{(r^n z)^\alpha}{\ell(r^n z)} \overline F(r^n z) = p(z)
\text{ for each  } z \in C_p, & \beta \in (0,1), \\
\lim_{n \to \infty} \ell(r^n)^{-1} \int_0^{r^n z} y^{m+1} \dd F(y) = p(z) 
\text{ for each  } z \in C_p, & \beta = 1.
\end{cases}
\label{eq:Ltail-3}
\end{align}
If $\beta > 0$, then (\ref{eq:Ltail-1})--(\ref{eq:Ltail-3}) are further equivalent to
\begin{equation} \label{eq:Ltail-4}
(-1)^{m+1} \widehat F^{(m+1)}(s) \sim s^{\beta -1} \ell(1/s) q_{m+1}(s)
\text{ as } s \downarrow 0, 
\end{equation}
and $q_{m+1} = \mathrm{B}_{\beta}^{-1} q_m$.

Moreover, the relations between the appearing functions are the following:
\[
\begin{split}
& q_m = \mathrm{B}_{\alpha- (m-1)}^{-1} \mathrm{B}_{\alpha - (m-2)}^{-1} \ldots 
\mathrm{B}_\alpha^{-1} \widetilde q_m, \ \beta \in [0,1], \ 
q_0 = \widetilde q_0, \\
& p_{0,m} = \mathrm{B}^{-1}_{1-\beta} \mathrm{A}_{1-\beta}^{-1} q_m, \ \beta \in (0,1), \\
& p = p_{0,m} - m \mathrm{B}_{-m-\beta} p_{0,m}, \ 
p_{0,m} = p + m \mathrm{B}_{-\beta} p, \ \beta \in (0,1).
\end{split}
\]
If $\beta \in \{0, 1\}$, then necessarily $p(x) \equiv p > 0$, $q_m(s) \equiv q_m > 0$, 
and $p = q_m$.
\end{theorem}

Since $p(x)$ is constant for $\beta \in \{0,1\}$, by Lemma \ref{lemma:cont-q}
(\ref{eq:Ltail-3}) is further equivalent to
$\int_{x}^\infty y^m \dd F(y) \sim p \ell(x)$, and 
$\int_0^{x} y^{m+1} \dd F(y) \sim p \ell(x)$ as $x \to \infty$, respectively.

\begin{proof}[Proof of Theorem \ref{thm:L-tail}]
We follow the proof of Theorem 8.1.8 in \cite{BGT}.

The equivalence of (\ref{eq:Ltail-1}) and (\ref{eq:Ltail-2}) follows from iterated 
application of Proposition \ref{prop:Karamata0}. 
(Note that the derivatives of $f_m$ are monotone.)
We obtain that $q_m = \mathrm{B}_{\alpha- (m-1)}^{-1} 
\mathrm{B}_{\alpha - (m-2)}^{-1} \ldots \mathrm{B}_\alpha^{-1} \widetilde q_m$.
Furthermore, for $\beta > 0$ both (\ref{eq:Ltail-1}) and (\ref{eq:Ltail-2}) are 
equivalent to (\ref{eq:Ltail-4}), and $q_{m+1} = \mathrm{B}_{\beta}^{-1} q_m$.

Put
\[
U_m(x) = \int_0^x \int_t^\infty y^m \dd F(y) \dd t, 
\]
and note that $\widehat U_m(s) = g_m(s) / s$. Therefore (\ref{eq:Ltail-2}) is equivalent 
to
\begin{equation} \label{eq:hatU1}
\widehat U_m(s) \sim s^{\beta -1} \ell(1/s) q_m(s).
\end{equation}
For $\beta \in [0,1]$, using Theorem \ref{thm:taub} with $\rho = 1 - \beta$ this is 
equivalent to
\begin{equation} \label{eq:U1}
\lim_{n \to \infty} \frac{U_m(r^n z)}{(r^n z)^{1-\beta} \ell(r^n z)} = p_m(z)
\quad z \in C_{p_m}, 
\end{equation}
where $p_m = \mathrm{A}_{1-\beta}^{-1} q_m$, for $\beta \neq 1$, and
$p_m = q_m$ for $\beta = 1$.

First assume $\beta \in (0,1)$.
By Theorems \ref{thm:Karamata} and \ref{thm:mondens} with $\rho = 1 - \beta$, the latter 
holds if and only if
\begin{equation} \label{eq:u1}
\lim_{n \to \infty} \frac{u_m(r^n z)}{(r^n z)^{-\beta} \ell(r^n z) } = p_{0,m}(z) 
\quad z \in C_{p_{0,m}},
\end{equation}
where $u_m(x) = \int_x^\infty y^m \dd F(y)$, and $\mathrm{B}_{1-\beta} p_{0,m} = p_m$. 
Note that for $m=0$ this is exactly (\ref{eq:Ltail-3}). Partial integration gives
\begin{equation} \label{eq:F1}
u_m(x) = x^m \overline F(x) + m \int_x^\infty y^{m-1} \overline F(y) \dd y. 
\end{equation}
If (\ref{eq:Ltail-3}) holds, then by Proposition \ref{prop:rho<0} with $\rho = -\beta$, 
we obtain 
(\ref{eq:u1}) with
$p_{0,m} = p + m \mathrm{B}_{-\beta} p$, so (\ref{eq:Ltail-2}) follows.

Conversely, using Fubini's theorem, we get
\begin{equation} \label{eq:F2}
 \frac{x^m \overline F(x)}{u_m(x)} = 1 - \frac{m x^m}{u_m(x)} \int_x^\infty y^{-m-1} 
u_m(y) \dd y.
\end{equation}
Now, Proposition \ref{prop:rho<0} with $\rho = - m - \beta$ shows that (\ref{eq:u1}) is 
further equivalent to
\begin{equation} \label{eq:u2}
\lim_{n \to \infty} \frac{\int_{r^n z}^\infty y^{-m-1} u_m(y) \dd y}{(r^n z)^{-m-\beta} 
\ell(r^n z)} = {\mathrm{B}}_{-m-\beta} p_{0,m} (z).
\end{equation}
Thus, if (\ref{eq:u1}) holds, then by (\ref{eq:F2})
\[
\lim_{n \to \infty}
\frac{(r^n z)^{m+\beta}}{\ell(r^n z)} \overline F(r^n z) =
p_{0,m}(z) - m \mathrm{B}_{-m-\beta} p_{0,m}(z) \quad z \in C_{p_{0,m}},
\]
which is exactly (\ref{eq:Ltail-3}).
\smallskip

For $\beta = 0$ conditions (\ref{eq:U1}) and (\ref{eq:u1}) are still equivalent. If 
(\ref{eq:u1}) holds, then the monotonicity of $u$ forces that $p_{0,m}$ is constant, thus 
(\ref{eq:Ltail-3}) follows with $p = p_{0,m}$. The converse is obvious.

For $\beta = 1$ note that $(-1)^{m+1} \widehat F^{(m+1)}(s)$ is the 
Laplace--Stieltjes transform of $\int_0^x y^{m+1} \dd F(y)$. Therefore, by Theorem 
\ref{thm:taub}, (\ref{eq:Ltail-3}) and (\ref{eq:Ltail-4}) are equivalent, and 
$q_{m+1} = p$.
\end{proof}

We spell out this result in the most important special case, when $m = 0$.
In this case $f_0(s) = g_0(s) = 1 - \widehat F(s)$.

\begin{corollary} \label{cor:alpha}
Let $\ell \in \mathcal{SV}_\infty$, 
$\alpha \in [0,1]$, $q_0, p \in \mathcal{P}_{r}$.
The following are equivalent:
\begin{align} 
& 1 - \widehat F(s) \sim s^\alpha \ell(1/s) q_0(s) \text{ as } s \downarrow 0; 
\label{eq:Ltail-1spec}\\
& 
\begin{cases}
\lim_{n \to \infty} \frac{(r^n z)^\alpha}{\ell(r^n z)} \overline F(r^n z) = p(z)
\text{ for each  } z \in C_p, & \alpha \in [0,1), \\
\lim_{n \to \infty} \ell(r^n)^{-1} \int_0^{r^n z} y \dd F(y) = p(z) 
\text{ for each  } z \in C_p, & \alpha = 1.
\end{cases}
\label{eq:Ltail-3spec}
\end{align}
If $\alpha > 0$, then (\ref{eq:Ltail-1spec}), (\ref{eq:Ltail-3spec}) are further 
equivalent to
\begin{equation} \label{eq:Ltail-4spec}
- \widehat F'(s) \sim s^{\alpha -1} \ell(1/s) q_{1}(s)
\text{ as } s \downarrow 0, 
\end{equation}
and $q_1 = \mathrm{B}_{\alpha}^{-1} q_0$.

Moreover, $p = \mathrm{B}^{-1}_{1-\alpha} \mathrm{A}_{1-\alpha}^{-1} q_0$, if
$\alpha \in (0,1)$. If $\alpha \in \{0, 1\}$, then necessarily $p(x) \equiv p > 0$, 
$q_0(s) \equiv q_0 > 0$, and $p = q_0$. 
\end{corollary}

\begin{example} \label{ex:StP}
\emph{St.~Petersburg distribution.}
The random variable $X$ has generalized St.~Petersburg distribution with 
parameter $\alpha \in (0,1]$ (and $p=q=1/2$) if 
$\p \{ X = 2^{n/\alpha} \} = 2^{-n}$, $n=1,2,\ldots$. The tail of the 
distribution function
\[
\overline F(x) = \p \{ X > x \} =
\frac{2^{\{ \alpha \log_2 x \}}}{x^\alpha}, \quad x \geq 2^{1/\alpha},
\]
where $\{ x \}$ stands for the fractional part of $x$.
On generalized St.~Petersburg distributions we refer to Cs\"org\H{o} 
\cite{Csorgo}, Berkes, Gy\"orfi, and Kevei \cite{BGyK}, and the references 
therein.

With the notation of Corollary \ref{cor:alpha}, for $\alpha < 1$
we have $r=2^{1/\alpha}$, $p(z) \equiv 2^{\{ \alpha \log_2 z\}}$, and $\ell(x) 
\equiv 1$, while if $\alpha = 1$ then $r=2$, $p(z) \equiv 1$, and $\ell(x) = 
\log_2 x$. In this special case for the Laplace 
transform
\[
\widehat F(s) = \sum_{n=1}^\infty e^{-2^{n/\alpha} s} 2^{-n}
\]
explicit computation shows that
\[
\begin{split}
1 - \widehat F(s) & \sim s^\alpha 
\sum_{m=- \infty}^\infty \left( 1- \exp \left[ 2^{(m - \{ \alpha \log_2 s^{-1} 
\} )/\alpha} \right] \right) 2^{-m + \{ \alpha \log_2 s^{-1} \} } \\
& =: s^\alpha q_0(s) 
\end{split} 
\]
as  $s \downarrow 0$, whenever $\alpha < 1$, and 
\[
1 - \widehat F(s) \sim s \log_2 s^{-1} \quad \text{as } s \downarrow 0,
\]
for $\alpha = 1$. This is exactly  the statement of Corollary \ref{cor:alpha}. 
A somewhat lengthy but straightforward calculation shows that
$q_0 = A_{1-\alpha} B_{1- \alpha} p$ for $\alpha < 1$.
\end{example}

\subsection{Fixed points of smoothing transforms} \label{subsect:smooth}

Let $T = (T_i)_{i \in \N}$ be a sequence of nonnegative random variables; it can be 
finite, or infinite, dependent, or independent. A random variable $X$, or its 
distribution, is the fixed point of the (homogeneous) smoothing transform corresponding 
to $T$, if
\begin{equation} \label{eq:smoothing-eq}
X \stackrel{\mathcal{D}}{=} \sum_{i \geq 1} X_i T_i, 
\end{equation}
where on the right-hand side $X_1, X_2, \ldots$ are iid copies of $X$, and they are 
independent of $T$. 

The theory of smoothing transforms goes back to Mandelbrot \cite{Mandelbrot}.
Existence and behavior of the solution of equations of type (\ref{eq:smoothing-eq}) was 
investigated by Durrett and Liggett \cite{DL}, Guivarc'h \cite{Gui}, Liu \cite{Liu},
Jelenkovi\'c and Olvera-Cravioto \cite{JOC3}, 
Alsmeyer, Biggins, and Meiners \cite{ABM},
to mention just a few. For applications and references we 
refer to Section 5.2 in the monograph \cite{BDM} by Buraczewski, Damek, and Mikosch.

Most of the results on the tail behavior of the solution provide 
conditions which imply exact power-law tail. We are aware of very few exceptions. Theorem 
2.2 in \cite{Liu} states that in the arithmetic case, under appropriate conditions there 
is an $\alpha > 0$, such that
\[
0 < \liminf_{x \to \infty} x^\alpha \p \{ X > x \} \leq 
\limsup_{x \to \infty} x^\alpha \p \{ X > x \} < \infty.
\]
Guivarc'h \cite[p.268]{Gui} noted without proof that in the arithmetic case under 
appropriate conditions the tail of $X$, the solution of (\ref{eq:smoothing-eq}) behaves as 
$q(x) x^{-\alpha}$, for some $p \in \mathcal{P}_{r, \alpha}$. The implicit renewal theory 
for the smoothing transform was worked out by Jelenkovi\'c and Olvera-Cravioto \cite{JOC3} 
both in the arithmetic case and nonarithmetic case.

In order to state the main result in \cite{ABM} we need some further definition 
and assumptions. Let $N = \sum_{i} I(T_i > 0)$ denote the number of positive 
terms in the 
right-hand side in (\ref{eq:smoothing-eq}) and put $m(\theta) = \E \sum_{i=1}^N 
T_i^\theta $. Assume that 
\begin{itemize}
\item[(i)] $\p \{ T \in \{ 0, 1\}^\N \} < 1$;
\item[(ii)] $\E N > 1$; 
\item[(iii)] there exists an $\alpha \in (0,1]$, such that $1 = m(\alpha) < m(\beta)$, 
for each $\beta \in [0,\alpha)$;
\item[(iv)] either
$\E \sum_{i \geq 1} T_i^\alpha \log T_i \in ( -\infty, 0)$ and 
$\E (\sum_{i \geq 1} T_i^\alpha ) \log_+ \sum_{i \geq 1} T_i^\alpha < \infty$,
or there is a $\theta  \in [0, \alpha)$, such that $m(\theta) < \infty$;
\item[(v)]
there exists a nonnegative random variable $W$, which is not identically 0, such 
that 
\[
W \stackrel{\mathcal{D}}{=} \sum_{i \geq 1} T_i^\alpha W_i, 
\]
where on the right-hand side $W_1, W_2, \ldots$ are iid copies of $W$, they are 
independent of $T$, and $T$ has the same distribution as in (\ref{eq:smoothing-eq});
\item[(vi)] the positive elements of $T$ are concentrated on $r^\Z$ for some $r > 1$, and 
$r$ is the smallest such number.
\end{itemize}
Under the above assumptions in \cite[Corollary 2.3]{ABM} it was showed that the Laplace 
transform $\varphi$ of the solution of the fixed point equation (\ref{eq:smoothing-eq}) 
has the form
\begin{equation} \label{eq:sol-smtrf}
\varphi(t) = \psi( h(t) t^\alpha), \quad t \geq 0,
\end{equation}
where $\alpha \in (0,1]$, $h$ is a logarithmically $r$-periodic function such that $h(t) 
t^\alpha$ is a Bernstein-function, i.e.~its derivative is completely monotone, and
$\psi$ is a Laplace transform of the random variable $W$ in (v), such that
$(1- \psi(t)) t^{-1}$ is 
slowly varying at 0. 

The tail behavior of the solutions was not discussed. Theorem 
\ref{thm:L-tail}, in particular Corollary \ref{cor:alpha}, allows us the determine 
the tail behavior of such solutions.
Indeed, if $\ell \in \mathcal{SV}_\infty$, then $\tilde \ell (x) = \ell (x^\alpha h(x)) 
\in \mathcal{SV}_\infty$. Therefore, from (\ref{eq:sol-smtrf}) 
\[
1 - \varphi(t) = t^\alpha \tilde \ell (1/t) h(t), 
\]
which allows us to apply Corollary \ref{cor:alpha}. Noting that 
$\ell (x^\alpha h(x)) \sim \ell(x^\alpha)$ as $x \to \infty$, we obtain the following.

\begin{corollary} \label{cor:smooth}
Assume (i)--(vi). If $\alpha < 1$, then for the tail $\overline F(x) = \p \{ X > x \}$ of 
the solution of equation 
(\ref{eq:smoothing-eq}) we have
\[
\lim_{n \to \infty} \frac{(r^n z)^\alpha}{\ell(r^{\alpha n})} \p \{ X > r^n z\} = p(z), 
\quad z \in C_p,
\]
where $p = \mathrm{B}^{-1}_{1-\alpha} \mathrm{A}^{-1}_{1-\alpha} h$. While, if 
$\alpha=1$, then $h(t) \equiv h$ is necessarily a constant, and  
\[
\int_0^{x} y \dd F(y) \sim  h\, \ell(x).
\]
\end{corollary}

\subsection{Semistable laws} \label{subsect:semi}

Logarithmically periodic functions, and regularly log-periodic functions naturally arise 
in the analysis of semistable distributions. The class of semistable laws, introduced by 
Paul L\'evy, is an important subclass of infinitely divisible laws. The semistable laws 
are the stable laws, and those infinitely divisible distributions, which has no normal 
component, and the L\'evy measure $\mu$ in the L\'evy--Khinchin representation satisfies
\[
\mu((x, \infty)) = x^{-\alpha} p_+(x), \ 
\mu((-\infty, x)) = x^{-\alpha} p_-(x), \ x > 0, 
\]
where $\alpha \in (0,2)$, $r > 1$, and $p_+, p_- \in \mathcal{P}_{r,-\alpha} \cup \{ 0 
\}$ (0 is the identically 0 function), such that at least one of them is not identically 
0.
For properties, characterization, applications and some history of semistable laws we 
refer to Megyesi \cite{Megyesi}, Huillet, Porzio, and Ben Alaya \cite{Huillet}, 
and 
Meerschaert and Scheffler \cite{MeerschaertS}, and the references therein. We note that in 
the characterization of the domain of geometric 
partial attraction regularly log-periodic functions play an important role; see Grinevich 
and Khokhlov \cite{GrinevichKhokhlov}, and Megyesi \cite{Megyesi}.

Although there has been large interest in semistable laws in the last 50 years, the tail 
behavior was determined only in 2012 by Watanabe and Yamamuro \cite{WY3}. We reprove some 
of their results, emphasizing that more precise and more general results were 
shown in 
\cite{WY3}. In particular, we restrict ourselves to the nonnegative semistable laws, 
since the technique developed in this paper works only for one-sided laws. 

The Laplace transform of a nonnegative semistable random variable $W$ has the form
\begin{equation} \label{eq:Laplace-sstable}
\E e^{-s W} = \exp \left\{ - a s - \int_0^\infty (1 - e^{-sy}) \nu(\dd y) \right\},
\end{equation}
where $a \geq 0$, and $\nu$ is a L\'evy measure such that
$\overline \nu(x) = p(x) x^{-\alpha}$, with $p \in \mathcal{P}_{r,-\alpha}$, $\alpha 
\in (0,1)$, and $\overline \nu(x) = \nu ((x, \infty))$, $x > 0$. Partial integration gives
\[
\int_0^\infty ( 1 - e^{-sy} ) \nu (\dd y) = \int_0^\infty e^{-sy} s \overline \nu(y) \dd y
= s \widehat U(s), 
\]
where 
\[
U(x) = \int_0^x \overline \nu(y) \dd y = x^{1-\alpha} \mathrm{B}_{1 - \alpha} p(x).
\]
From Theorem \ref{thm:taub} we have
\[
\widehat U(s) \sim s^{\alpha - 1} q(s) \quad \text{as } s \downarrow 0,
\]
with $q= \mathrm{A}_{1-\alpha} \mathrm{B}_{1-\alpha} p$. Thus, 
(\ref{eq:Laplace-sstable}) gives
\[
1 - \E e^{-s W} \sim a s + \int_0^\infty (1 - e^{-sy}) \nu(\dd y) \sim 
s^{\alpha} q(s)
\quad \textrm{as } s \downarrow 0.
\]
Corollary \ref{cor:alpha} implies
\[
\lim_{n \to \infty} (r^n z)^\alpha \, \p \{ W > r^n z \} = p(z)
\quad \text{for each } z \in C_p, 
\]
or, which is the same
\[
\lim_{n \to \infty} r^{n \alpha} \, \p \{ W > r^n z \} = \overline \nu(z)
\quad \text{for each } z \in C_{p}.  
\]
This is the statement in Theorem 1 \cite{WY3}. However, there the limit above is 
determined for any $z > 0$.

\subsection{Supercritical Galton--Watson process} \label{subsect:GW}

Consider a supercritical Galton--Watson process $(Z_n)_{n \in \N}$, $Z_0 =1$, with 
offspring generating function $f(s) = \E s^{Z_1}$, and offspring mean $\mu = \E Z_1 \in 
(1, \infty)$. Let $q \in [0,1)$ denote the extinction probability, i.e.~the smaller root 
of $f(s) = s$ in $[0,1]$. Denote $f_n$ the $n$-fold iterate of $f$, which is 
the generating function of $Z_n$. 
On general theory of branching processes see Athreya and Ney \cite{AthreyaNey}. 

Further assume $\E Z_1 \log Z_1 < \infty$, which assures that
\[
\frac{Z_n}{\mu^n} \longrightarrow W \quad \text{as } n \to \infty  \, \text{ a.s.}, 
\]
with $\E W = 1$. The Laplace transform of $W$, $\varphi(t) = \E e^{- t W}$, $t \geq 0$, 
satisfies the Poincar\'e functional equation
\begin{equation} \label{eq:poincare}
 \varphi( \mu t) = f(\varphi(t)).
\end{equation}
The latter equation always has a unique (up to scaling) solution, which is a Laplace 
transform of a distribution. However, the law of $W$ can be determined explicitly only in 
very few special cases. Therefore, it is important to obtain asymptotic behavior of the 
tail probabilities.  Assume that we are in the Schr\"oder case, that is $\gamma= f'(q) 
> 0$. Then the Schr\"oder function
\[
Q(s) = \lim_{n \to \infty} \frac{f_n(s) - q}{ \gamma^n} 
\]
is continuous and it is the unique solution of the functional equation
\begin{equation} \label{eq:schroeder}
Q(f(s)) = \gamma Q(s) \quad s \in [0,1),
\end{equation}
which satisfies $Q(q) = 0$ and $\lim_{s \to q} Q'(s) = 1$. Using the functional equations 
(\ref{eq:poincare}) and (\ref{eq:schroeder}) one has that the \emph{Karlin--McGregor 
function} of $f$
\[
K(s) = s^\alpha Q(f(s)),
\]
with $\alpha = - \log \gamma / \log \mu$, is logarithmically periodic with period $\mu$.
Harris \cite[Theorem 3.3]{Harris} proved that
\begin{equation} \label{eq:harris}
\varphi(s) \sim \frac{K(s)}{s^\alpha} \quad \text{as } \, s \to \infty. 
\end{equation}
From a version of Theorem \ref{thm:taub}, with $n \to - \infty$ in 
(\ref{eq:U-asy}) and $s \to \infty$ in (\ref{eq:hatU-asy}), it follows for the 
distribution function
$G(x) = \p \{ W \leq x \}$ that
\begin{equation} \label{eq:F-Harris}
\lim_{n \to \infty} G(r^{-n} z) (r^{-n} z)^{-\alpha} =  p(z),
\end{equation}
with $p = \mathrm{A}^{-1}_\alpha K$.

A much stronger result was shown by Biggins and Bingham \cite[Theorem 4]{BigginsBingham}, 
namely
\[
G'(x) \sim x^{\alpha - 1} V(x) \quad \text{as } \, x \downarrow 0,
\]
where $V$ is a continuous, positive, logarithmically periodic function with period $\mu$.
For further results on tail asymptotics of $W$ we refer to Bingham \cite{Bingham}, 
Biggins and Bingham \cite{BigginsBingham}, and to the more recent papers by Fleischmann 
and Wachtel \cite{FW} and by Wachtel, Denisov, and Korshunov \cite{WDK}.

\bigskip

\noindent
\textbf{Acknowledgement.}
This research was supported by the J\'anos Bolyai Research Scholarship of the Hungarian 
Academy of Sciences, and by the NKFIH grant FK124141.

\def\cprime{$'$}

\end{document}